\documentclass[a4paper,11pt,leqno]{amsart}
\usepackage{a4wide} 
\usepackage{amscd}
\usepackage{amssymb}
\usepackage{amsfonts}
\usepackage[centertags]{amsmath}
\usepackage{latexsym}
\usepackage{verbatim}
\usepackage{amsthm}
\usepackage{color}
\usepackage{pgfplots}
\usepackage{tikz}
\usetikzlibrary{intersections}
\usepackage{amsfonts}
\usepackage{amssymb}
\usepackage{amsmath}
\usepackage{amsthm}
\usepackage[english]{babel}
\usepackage{subfigure}
\usepackage{enumerate}
\usepackage{epsfig}
\usepackage{graphicx}
\usepackage{color}
\usepackage{hyperref}
\pdfoutput=1

\newtheorem{theorem}{Theorem}[section]
\newtheorem{corollary}[theorem]{Corollary}
\newtheorem{lemma}[theorem]{Lemma}

\newtheorem{ackn}{Acknowledgments\!}

\theoremstyle{remark}

\DeclareMathOperator{\tr}{\mathrm{tr}}

\title[On the umbilic set of immersed surfaces]{On the umbilic set of immersed surfaces\\ in three-dimensional space forms}

\author{Giovanni Catino, Alberto Roncoroni, Luigi Vezzoni}

\date{\today}
\thanks{}
\address{G. Catino, Dipartimento di Matematica, Politecnico di Milano, Piazza Leonardo da Vinci 32, 20133 Milano, Italy.} \email{giovanni.catino@polimi.it} 

\address{A. Roncoroni,  Dipartimento di Matematica F. Casorati, Universit\`a di Pavia, Via Ferrata 5, 27100 Pavia, Italy.} \email{alberto.roncoroni01@universitadipavia.it} 

\address{L. Vezzoni, Dipartimento di Matematica G. Peano, Universit\`a di Torino, Via Carlo Alberto 10, 10123 Torino, Italy.} \email{luigi.vezzoni@unito.it}


\begin{document}

\maketitle

\begin{abstract}
We prove that under some assumptions on the mean curvature the set of umbilical points of an immersed surface in a $3$-dimensional space form has positive measure.  In  case of an immersed sphere our result can be seen as a generalization of the celebrated Hopf theorem. 
\end{abstract}

\

\begin{center}


\medskip

\centerline{\bf AMS subject classification:  53C40, 53C42, 53A10}

\end{center}


%
%
%

\

\section{Introduction}
In \cite{Hopf1} Hopf proved his famous theorem 

 \begin{center}
 {\em  An immersed sphere of constant mean curvature in $\mathbb{R}^{3}$ is a round
sphere.}
\end{center}
The proof of Hopf's theorem involves the so-called {\em Hopf differential} $Q$, which is defined as the $(2,0)$-component of the second fundamental form of the surface (once regarded as a Riemann surface). The key point is that the zeros of $Q$ are the umbilical points of the surface and if the mean curvature is constant, $Q$ is holomorphic.  Hence the set of umbilical points of a connected constant mean curvature surface can be only discrete or the whole surface. In the case of an immersed sphere Hopf showed 
that this set cannot be discrete and the theorem follows. 

Hopf's theorem was generalized in three-dimensional space forms by Chern in \cite{Chern} and, more recently, in $3$-dimensional homogeneous spaces (see \cite{meeks} and the references therein). The study of constant mean curvature spacial Riemannian manifolds is a central subject in differential geometry and there are many interesting results on
this topic (see e.g. \cite{AdC,AL,Brendle,Catino1,Catino2,CdCK,NS,Simons,Yau} and the reference therein).

Moreover, it is well known that the theorem cannot be generalized to surfaces of higher genus and the first counterexample was provided by Wente in \cite{Wente} who constructed an immersed torus in $\mathbb{R}^3$ having constant mean curvature.

In order to state the result of the present paper we fix some notation for a given surface $M$ immersed in the $3$-dimensional space form $\mathbb{F}^3(c)$ of curvature $c$: $h$ is the second fundamental form; $H={\rm tr}\,h$ is the mean curvature; $g$ is the induced metric on $M$; $\mathring{h}=h-\frac{1}{2}\tr(h)g$ is the trace-free part of $h$. Moreover for $\epsilon>0$ we define the set
$$
\Omega_{\epsilon}=\{|\mathring{h}|\geq \epsilon\}\,.
$$ 
Notice that by definition $\Omega_{0}^c$ is the set of umbilical points of $M$. 

\

The aim of this note is provide some sufficient conditions on the mean curvature which imply that the set of umbilical points $\Omega_0^c$ of an immersed surface in a $3$-dimensional space form has positive measure.  More precisely the result reads as follows 
\begin{theorem}\label{Teo_Principale}
Let $M$ be an immersed surface in $\mathbb{F}^{3}(c)$. Then, for every $0<\varepsilon\leq1$, there exists a positive constant $C$, depending only on $||H||_{\infty}$ and $c$, such that
\begin{equation}\label{prel}
C\, {\rm Vol}(\Omega^c_\varepsilon) \geq  \frac{2}{\varepsilon^4}\int_{\Omega^c_\varepsilon}|\nabla \mathring{h}|^2|\mathring{h}|^2 - \frac{1}{\varepsilon^4}\int_{\Omega^c_\varepsilon}|\nabla H|^2|\mathring{h}|^2 +4\pi\chi(M)\, .
\end{equation}
In particular, if
\begin{equation}\label{ipo}
\limsup_{\varepsilon\rightarrow 0}\frac{1}{\varepsilon^4}\int_{\Omega^c_\varepsilon}|\nabla H|^2|\mathring{h}|^2  < \limsup_{\varepsilon\rightarrow 0}\frac{2}{\varepsilon^4}\int_{\Omega^c_\varepsilon}|\nabla \mathring{h}|^2|\mathring{h}|^2 +4\pi\chi(M) \,,
\end{equation}
then ${\rm Vol}(\Omega^c_0)>0.$
\end{theorem}
As it will be clear from the proof, the result is sharp. That can be also deduced by checking that an ellipsoid of revolution (not spherical), which has two umbilical points, realizes the equality in \eqref{ipo} with $\chi(M)=2$. 

As an immediate corollary we have
\begin{corollary}\label{Cor}
Let $M$ be an immersed sphere in $\mathbb{F}^{3}(c)$. Then ${\rm Vol}(\Omega^c_0)>0$ if one of the following holds 
\begin{enumerate}   

\item[$1.$] $H$ is constant on $\Omega^c_{\varepsilon_0}$, for some $\varepsilon_0>0$;

\vspace{0.1cm}
\item[$2.$] $|\nabla H|^2 \leq 2 |\nabla \mathring{h}|^2$ on $\Omega^c_{\varepsilon_0}$, for some $\varepsilon_0>0$;

\vspace{0.1cm}
\item[$3.$] $\nabla H$ satisfies
$$
\limsup_{\varepsilon\rightarrow 0}\frac{1}{\varepsilon^2}\int_{\Omega^c_\varepsilon}|\nabla H|^2 < 8\pi\, .
$$
\end{enumerate}
\end{corollary}
Notice that the corollary in particular implies that if $H$ is constant, then the set of umbilical points on $M$ has positive measure which implies Hopf's theorem.     

\

\noindent 
{\bf Notation.} Throughout the paper the Einstein convention will be adopted and and the summation over repeated indices is omitted.

\

\section{Proof of Theorem \ref{Teo_Principale} and Corollary \ref{Cor}}

As preliminary result we prove a Bochner-Weitzenb\"ock type formula for the trace-free part of the second fundamental form of an immersed surface in  a three dimensional space forms (for similar results we refer to \cite{Bochner,Bourguignon}).

\begin{lemma}
Let $M$ be an immersed surface in  $\mathbb F^3(c)$. Then
\begin{align}\label{Laplacian_quater}
\frac{1}{2}\Delta |\mathring{h}|^2 |\mathring{h}|^2=
2|\nabla |\mathring{h}||^2|\mathring{h}|^2-\mathring{h}_{ij}\nabla_i|\mathring{h}|^2\nabla_j H+\frac{1}{2}|\nabla H|^2|\mathring{h}|^2+R|\mathring{h}|^4+ \mathring{h}_{ij} \nabla_i\nabla_j H|\mathring{h}|^2 \,,
\end{align}
where  $R$ is the the scalar curvature of $M$.
\end{lemma}
\begin{proof}
For any $p\in M$ we choose a local orthonormal frame $\{e_1,e_2,e_3\}$ in $\mathbb F^3(c)$ around $p$ such that $\{e_1,e_2\}$ are tangential to $M$. In this frame Codazzi and Gauss equations read as (see e.g. \cite{DoCarmo})
\begin{equation}\label{codazzi_eq}
\nabla_k h_{ij}-\nabla_j h_{ik}=0\, ,
\end{equation}
and the Riemannian curvature tensor of $M$ satisfies
\begin{equation}\label{gauss_eq}
R_{ikjl} = c (g_{ij}g_{kl} + g_{il}g_{jk}) + h_{ij}h_{kl} - h_{il}h_{jk}\, . 
\end{equation}
From \eqref{codazzi_eq} we get that $\mathring{h}=h-(H/2) g$ satisfies the following equation 
\begin{equation}\label{codazzi_eq_trace}
\nabla_k \mathring{h}_{ij}-\nabla_j \mathring{h}_{ik}=\frac{1}{2}\left(\nabla_j H g_{ik}-\nabla_k H g_{ij}\right)\, .
\end{equation}
Taking the divergence of \eqref{codazzi_eq_trace} we get
\begin{equation}\label{div}
\nabla_k \mathring{h}_{ij}=\frac{1}{2}\nabla_j H\,.
\end{equation}
Moreover, taking the covariant derivative of \eqref{codazzi_eq_trace} and tracing we obtain
\begin{align}\label{Laplacian}
\Delta \mathring{h}_{ij}=&\nabla_j\nabla_k\mathring{h}_{ik}+R\,\mathring{h}_{ij}+\frac{1}{2}\left(\nabla_i\nabla_j H-\Delta Hg_{ij}\right) \, .
\end{align}
From \eqref{div} and the commutation formula 
$$
\nabla_k\nabla_j\mathring{h}_{ik} = \nabla_j\nabla_k\mathring{h}_{ik} +R\,\mathring{h}_{ij}\, ,
$$ 
we deduce 
\begin{align}\label{Laplacian_bis}
\Delta \mathring{h}_{ij}=R\,\mathring{h}_{ij}+\nabla_i\nabla_j H-\frac{1}{2}\Delta Hg_{ij} \, .
\end{align}
From \eqref{Laplacian_bis} we obtain 
\begin{equation}\label{Laplacian_bis_bis}
\frac{1}{2}\Delta |\mathring{h}|^2=|\nabla\mathring{h}|^2+R|\mathring{h}|^2+\nabla_i\nabla_j H\mathring{h}_{ij}
\end{equation}
and then 
\begin{equation}\label{Laplacian_ter}
\frac{1}{2}\Delta |\mathring{h}|^2|\mathring{h}|^2=|\nabla\mathring{h}|^2|\mathring{h}|^2+R|\mathring{h}|^4+\nabla_i\nabla_j H\mathring{h}_{ij}|\mathring{h}|^2\, .
\end{equation}
The conclusion follows from the following formula: 
\begin{equation}\label{Smo}
2|\mathring{h}|^2\left(|\nabla \mathring{h}|^2-\frac{1}{2}|\nabla H|^2\right)=4|\nabla |\mathring{h}||^2|\mathring{h}|^2-2\mathring{h}_{ij}\nabla_i|\mathring{h}|^2\nabla_j H
\end{equation}
which follows from \cite[Formula 34]{Smoczyk} taking into account that $|\nabla h|^2=|\nabla\mathring{h}|^2+\frac{1}{2}|\nabla H|^2$; indeed by using \eqref{Smo} in \eqref{Laplacian_ter} we get \eqref{Laplacian_quater}.

\end{proof}

\

\noindent Now we can prove Theorem \ref{Teo_Principale}.

\begin{proof}[Proof of Theorem $\ref{Teo_Principale}$]
We consider the following positive function
$$
f_\varepsilon=\begin{cases}
|\mathring{h}|\ & \text{ in } \Omega_\varepsilon \\
\varepsilon & \text{ in } \Omega^c_\varepsilon  \, ,
\end{cases}
$$
where, we recall that, $\Omega_\varepsilon=\{|\mathring{h}|\geq\varepsilon\}$. By multiplying \eqref{Laplacian_quater} by $f^{-4}_\varepsilon$ and integrating over $M$  we get
\begin{align}\label{integrale}
0=&-\frac{1}{2}\int_M\Delta |\mathring{h}|^2|\mathring{h}|^2f^{-4}_\varepsilon+2\int_M|\nabla |\mathring{h}||^2|\mathring{h}|^2f^{-4}_\varepsilon-\int_M\mathring{h}_{ij}\nabla_i|\mathring{h}|^2\nabla_j Hf^{-4}_\varepsilon\\\nonumber
&+\frac{1}{2}\int_M|\nabla H|^2|\mathring{h}|^2f^{-4}_\varepsilon +\int_M R|\mathring{h}|^4f^{-4}_\varepsilon+\int_M \nabla_i\nabla_j H\mathring{h}_{ij}|\mathring{h}|^2f^{-4}_\varepsilon\, .
\end{align}
Now we analyse the first, the second and the last integrals in \eqref{integrale}. About the first one: by integrating by parts and taking into account that $\nabla f_\varepsilon=0$ in $\Omega^c_\varepsilon$ we get
\begin{equation*}
-\frac{1}{2}\int_M\Delta |\mathring{h}|^2|\mathring{h}|^2f^{-4}_\varepsilon=\frac{1}{2}\int_M|\nabla|\mathring{h}|^2|^2f^{-4}_\varepsilon-4\int_{\Omega_\varepsilon}|\nabla|\mathring{h}||^2|\mathring{h}|^{-2}\ , 
\end{equation*}
i.e. 
\begin{equation*}
-\frac{1}{2}\int_M\Delta |\mathring{h}|^2|\mathring{h}|^2f^{-4}_\varepsilon=2\int_{\Omega_\varepsilon}|\nabla|\mathring{h}||^2|\mathring{h}|^2f^{-4}_\varepsilon+2\int_{\Omega^c_\varepsilon}|\nabla|\mathring{h}||^2|\mathring{h}|^2f^{-4}_\varepsilon-4\int_{\Omega_\varepsilon}|\nabla|\mathring{h}||^2|\mathring{h}|^{-2}\,  ,
\end{equation*}
where we have used that $|\nabla|\mathring{h}|^2|^2=4|\mathring{h}|^2|\nabla|\mathring{h}||^2$. From the definition of $f_\varepsilon$ we obtain
\begin{equation}\label{1}
-\frac{1}{2}\int_M\Delta |\mathring{h}|^2|\mathring{h}|^2f^{-4}_\varepsilon=-2\int_{\Omega_\varepsilon}|\nabla|\mathring{h}||^2|\mathring{h}|^{-2}+\frac{2}{\varepsilon^4}\int_{\Omega^c_\varepsilon}|\nabla|\mathring{h}||^2|\mathring{h}|^2\, .
\end{equation}
Moreover the definition of $f_\varepsilon$ implies
\begin{equation}\label{2}
2\int_M|\nabla |\mathring{h}||^2|\mathring{h}|^2f^{-4}_\varepsilon=2\int_{\Omega_\varepsilon}|\nabla|\mathring{h}||^2|\mathring{h}|^{-2}+\frac{2}{\varepsilon^4}\int_{\Omega^c_\varepsilon}|\nabla|\mathring{h}||^2|\mathring{h}|^2
\end{equation}
and by integrating by parts, using  the definition of $f_\varepsilon$ and taking into account that  $\nabla f_\varepsilon=0$ in $\Omega^c_\varepsilon$ we get
\begin{equation}\label{3}
\begin{aligned}
\int_M \nabla_i\nabla_j H\mathring{h}_{ij}|\mathring{h}|^2f^{-4}_\varepsilon=&-\frac{1}{2}\int_M|\nabla H|^2|\mathring{h}|^2f^{-4}_{\varepsilon}-\frac{1}{\varepsilon^4}\int_{\Omega^c_\varepsilon}\mathring{h}_{ij}\nabla_i|\mathring{h}|^2\nabla_j H \\
&+\int_{\Omega_\varepsilon}\mathring{h}_{ij}\nabla_j H\nabla_i|\mathring{h}|^2|\mathring{h}|^{-4}\, .
\end{aligned}
\end{equation}
So from \eqref{integrale}, \eqref{1}, \eqref{2}, \eqref{3} and from the definition of $f_\varepsilon$ we obtain
\begin{equation*}
0=\int_{\Omega_\varepsilon}R+\frac{1}{\varepsilon^4}\int_{\Omega^c_\varepsilon}R|\mathring{h}|^4-\frac{2}{\varepsilon^4}\int_{\Omega^c_\varepsilon}\mathring{h}_{ij}\nabla_i|\mathring{h}|^2\nabla_j H+\frac{4}{\varepsilon^4}\int_{\Omega^c_\varepsilon}|\nabla|\mathring{h}||^2|\mathring{h}|^2\, .
\end{equation*}
From Gauss-Bonnet Theorem we get
\begin{align}\label{integrale_tris}
0=&4\pi\chi(M)-\int_{\Omega^c_\varepsilon}R-\frac{2}{\varepsilon^4}\int_{\Omega^c_\varepsilon}\mathring{h}_{ij}\nabla_i|\mathring{h}|^2\nabla_j H+\frac{4}{\varepsilon^4}\int_{\Omega^c_\varepsilon}|\nabla|\mathring{h}||^2|\mathring{h}|^2+\frac{1}{\varepsilon^4}\int_{\Omega^c_\varepsilon}R|\mathring{h}|^4 \, .
\end{align}
Moreover, tracing Gauss equation \eqref{gauss_eq} we get $R=2c+H^2-|h|^2$, which in terms of $\mathring{h}$ becomes $R=\frac{1}{2}H^2-|\mathring{h}|^2+2c$; so \eqref{integrale_tris} yields 
\begin{align*}
0=&4\pi\chi(M)-\int_{\Omega^c_\varepsilon}\left(\frac{1}{2}H^2-|\mathring{h}|^2+2c\right)-\frac{2}{\varepsilon^4}\int_{\Omega^c_\varepsilon}\mathring{h}_{ij}\nabla_i|\mathring{h}|^2\nabla_j H \\\nonumber
&+\frac{4}{\varepsilon^4}\int_{\Omega^c_\varepsilon}|\nabla|\mathring{h}||^2|\mathring{h}|^2+\frac{1}{\varepsilon^4}\int_{\Omega^c_\varepsilon}\left(\frac{1}{2}H^2-|\mathring{h}|^2+2c\right)|\mathring{h}|^4\, .
\end{align*}
From $\eqref{Smo}$, we obtain
\begin{equation}
\begin{aligned}
0  = &\, 4\pi\chi(M)+\frac{2}{\varepsilon^4}\int_{\Omega^c_\varepsilon}|\nabla \mathring{h}|^2|\mathring{h}|^2 - \frac{1}{\varepsilon^4}\int_{\Omega^c_\varepsilon}|\nabla H|^2|\mathring{h}|^2 \\ \nonumber
&-\int_{\Omega^c_\varepsilon}\left(\frac{1}{2}H^2-|\mathring{h}|^2+2c\right)+\frac{1}{\varepsilon^4}\int_{\Omega^c_\varepsilon}\left(\frac{1}{2}H^2-|\mathring{h}|^2+2c\right)|\mathring{h}|^4 
\end{aligned}
\end{equation}
i.e. 
\begin{equation*}
0  \geq \, 4\pi\chi(M)+\frac{2}{\varepsilon^4}\int_{\Omega^c_\varepsilon}|\nabla \mathring{h}|^2|\mathring{h}|^2 - \frac{1}{\varepsilon^4}\int_{\Omega^c_\varepsilon}|\nabla H|^2|\mathring{h}|^2 - C\, {\rm Vol}(\Omega^c_\varepsilon)\, ,
\end{equation*}
which is \eqref{prel}, where $C$ is given by
$$
C=\frac12 \Vert H\Vert^2_\infty + 4|c| +1,
$$
if $\varepsilon \leq 1$. Now since $|\Omega^c_\varepsilon|\rightarrow|\Omega^c_0|$ as $\varepsilon\to 0$, then, if \eqref{ipo} holds, then ${\rm Vol}(\Omega^c_0)>0$. This concludes the proof of Theorem \ref{Teo_Principale}.

\end{proof}

\

\begin{proof}[Proof of Corollary $\ref{Cor}$]
The proof of Corollary $\ref{Cor}$ is an immediate application of Theorem \ref{Teo_Principale} for immersed spheres (i.e. $\chi(M)=2$), since each of the three conditions implies inequality \eqref{ipo}. 
\end{proof}

\

\begin{ackn}
\noindent {\em G.C and A.R. have been partially supported by the "Gruppo Nazionale per l'Analisi Matematica, la Probabilità e le loro Applicazioni" (GNAMPA) of the "Istituto Nazionale di Alta Matematica" (INdAM). L.V. has been partially supported by the G.N.S.A.G.A. of I.N.d.A.M. This manuscript was written while G.C. and A.R. were visiting the Department of Mathematics of the University of Turin, which is acknowledged for the hospitality.
}
\end{ackn}

\

\

\

\

\

\end{document}